\documentclass[11pt]{article}
\usepackage{amsmath,amsfonts,amsthm,amssymb,subfiles}

\usepackage{fullpage}

\usepackage{bbm}
\usepackage[english]{babel}
\usepackage[utf8x]{inputenc}
\usepackage[T1]{fontenc}
\usepackage{amsmath}
\DeclareMathOperator{\sign}{sign}
\newtheorem{theorem}{Theorem}[section]

\newtheorem{lemma}[theorem]{Lemma}
\newtheorem{cor}{Corollary}[section]

\makeatletter
\newtheorem*{rep@theorem}{\rep@title}
\newcommand{\newreptheorem}[2]{%
\newenvironment{rep#1}[1]{%
 \def\rep@title{#2 \ref{##1}}%
 \begin{rep@theorem}}%
 {\end{rep@theorem}}}
\makeatother

\newreptheorem{theorem}{Theorem}
\newreptheorem{cor}{Corollary}

\newreptheorem{lemma}{Lemma}

\theoremstyle{definition}
\newtheorem{definition}[theorem]{Definition}
\newtheorem{remark}[theorem]{Remark}

\title{On majorization for polynomials sharing a common interlacer}
\date{\today}
\author{Aurelien Gribinski\\ EPFL,Jussieu
}
\begin{document}

\maketitle
\begin{abstract}
We give necessary and sufficient conditions for majorization of realrooted polynomials sharing a common interlacer by means of residues coming from fraction decomposition. We also introduce a motivated notion called strong majorization, and we show how it relates to standard majorization.   
\end{abstract}
\section{Introduction}
The notion of majorization is fundamental in linear algebra. It has applications in many different fields, including convex geometry and probability (via doubly stochastic matrices). In this paper we focus on majorization for roots of polynomials, and we deduce some systematic criteria related to fraction decomposition  to check whether majorization is taking place. In particular we come up with a new notion called strong majorization that can be easily checked through partial fraction decomposition. We need to point out that majorization between polynomials gives a lot of information on the roots of one of the two polynomials if the roots of the other ones are known, as it is a strong property that involves all roots simultaneously. Majorization of polynomials was studied by  Borcea and Branden(in \cite{Bor}) who showed that linear operators preserving hyperbolicity also preserve majorization. Recently, Van Assche showed (in \cite{Assche}) that roots of consecutive orthogonal polynomials verify majorization.   \\

\begin{definition}[vector and polynomial majorization, from \cite{BMAJ}]
We say that two vectors $a=(a_1,a_2\dots,a_n)$ and $b=(b_1,b_2\dots,b_n)$ with $a_1\geq a_2\dots\geq a_n$  and  $b_1\geq b_2\dots\geq b_n$ are such that $a$ majorizes $b$ written  $a \succeq b $   if, for all $k \leq n$,
\begin{align*}
\sum_{i=1}^ka_i &\geq \sum_{i=1}^kb_i    &\text{and}&&  & \sum_{i=1}^na_i =\sum_{i=1}^nb_i .
\end{align*}
We say that a polynomial $p$ of degree $n$  with roots $\lambda=\big(\lambda_1,\lambda_2\dots,\lambda_n\big)$ where $\lambda_1\geq \lambda_2\dots \geq \lambda_n$ majorizes $q$ with roots $\mu=\big(\mu_1,\mu_2\dots,\mu_n\big)$  where also $\mu_1\geq \mu_2\dots \geq \mu_n$, denoted by $p \succeq q $,   if $ \lambda \succeq \mu$. 
\end{definition}

 In the following, we will interpolate continuously between two real-rooted polynomials $p$ and $q$. 
The property of common interlacing insures, as we will see in the next lemma,  that convex combinations of real-rooted polynomials are still real-rooted. It was extensively used in \cite{MSS1} to prove existence of Ramanujan graphs and in \cite{MSS2} to prove the Kadison Singer conjecture. 
\begin{definition}\label{interl}
 Two degree $n$ polynomials $p$ and $q$  have a common interlacer if there exists a third polynomial interlacing both $p$ and $q$. More explicitely, it happens when  the roots of $p$, denoted by $(\lambda_i)_{i=1}^n$ , with $\lambda_1 \geq  \lambda_2 \dots \geq \lambda_n$ and the roots of $q$ , denoted by $(\mu_i)_{i=1}^n$, with $ \mu_1 \geq  \mu_2 \dots \geq \mu_n$, are such that the intervals $[\lambda_i,\mu_i]$ or $[\mu_i,\lambda_i]$ are non-crossing,  depending on  $\lambda_i \leq \mu_i$ or $\mu_i \leq \lambda_i$ .  
\end{definition}
\begin{lemma}[Lemma 4.5 from \cite{MSS1}]\label{commoninterl}
$p$ and $q$ have a common interlacer $r$ if and only if $p_t:=tp+(1-t)q$ is real-rooted for all $t \in [0,1]$. Furthermore $r$  also interlaces $p_t$, and the roots $\lambda_i(t)$ of $p_t$ move monotonously in the intervals  $[\lambda_i,\mu_i]$ or $[\mu_i,\lambda_i]$.
\end{lemma}
In all the following we will assume that the roots of the polynomials considered are simple, and $p$ and $q$ have a common interlacer. Also, the vector of nonincreasing roots of $p$ will be denoted by $\lambda$ and the vector of roots of $q$ by $\mu$.

We can decompose the ratio $p$ over $q$ in simple poles:
\[
\frac{p}{q}= 1+ \sum_{i=1}^{i=n} \frac{p[\mu_i]}{q'[\mu_i]}\frac{1}{(x-\mu_i)}.
\]
Notice that by looking at the asymptotic coefficient in front of $\frac{1}{x}$ in the expansion on both sides and equating we get, assuming that majorization holds: 
\begin{equation}
 \sum_{i=1}^n\frac{p[\mu_i]}{q'[\mu_i]}= \sum_{i=1}^n\mu_i -  \sum_{i=1}^n\lambda_i=0,
\end{equation}
using the fact that the total sum of roots of vectors that majorize are equal. \\
Assume that $\lambda_i=\mu_i$ for $1\leq i\leq n$, then  $\frac{p[\mu_i]}{q'[\mu_i]}=0$ and for $j \neq i$,  $\frac{\tilde{p}[\mu_j]}{\tilde{q}'[\mu_j]}=\frac{p[\mu_i]}{q'[\mu_i]}$ where $\tilde{p}:= \frac{p}{x-\lambda_i}$ and $\tilde{q}:= \frac{q}{x-\mu_i} $. Therefore, we can replace $p$ and $q$ by $\tilde{p}$ and $\tilde{q}$ and the partial sums of residues will all be the same, and we get rid of a redundancy. Also, when looking at the roots $\lambda_i(t)$ of $tp+(1-t)q$ in subsequent proofs, if $\lambda_i=\mu_i$ then trivially $\lambda_i(t)=\lambda_i \forall t \in [0,1]$, so we can remove the shared roots and it won't affect the monotonicity properties considered.
 In the following we will assume (without loss of generality as we have just seen) that the roots of $p$ and $q$ are distinct.

\subsection{Summary of the results}

Section~2 of our paper gives a necessary condition for majorization in terms of sums of residues of the quotient of the two polynomials.

\begin{reptheorem}{NCM}[Necessary condition for majorization]
For polynomials $p$ and $q$ with a common interlacer and distinct roots,
\[
p \succeq q \implies \sum_{i=1}^{i=k} \frac{p[\mu_i]}{q'[\mu_i]} < 0 \text{  } \forall k<n.
\]
\end{reptheorem}
\begin{remark}
The property of "having a common interlacer" is much weaker than proper interlacing: in each interval, the order between roots of $p$ and $q$ can get reversed. It is a standard result from the theory of interlacing (see Section~2.3 in \cite{Wagner}) that interlacing of $p$ and $q$ is equivalent to the fact that all residues  $\frac{p[\mu_i]}{q'[\mu_i]}$ have the same sign. 
\end{remark}
As the condition is not sufficient, we need to find a stronger characterization. We come up in Section~3 with a sufficient condition by reversing the residues:

\begin{theorem}[Sufficient condition, follows directly from Theorem~\ref{suffcor}]
For polynomials $p$ and $q$ with a common interlacer, assuming that  $\sum_{i=1}^n\lambda_i= \sum_{i=1}^n\mu_i$,
\[
 \sum_{i=1}^{i=k} \frac{q[\lambda_i]}{p'[\lambda_i]} \geq 0 \text{ }\forall k<n  \implies  p \succeq q .
\]
\end{theorem}

 It so happens that to find a sufficient characterization of majorization in terms of residues, we were led to strengthen the notion of majorization, and consider a continuous version. 

\begin{definition}[Strong majorization]
Assume $p$ and $q$ have a common interlacer as usual. Denote by $ \lambda_i(t)$ the roots of $t p+(1-t)q$ by decreasing order. We say that $p$ strongly majorizes $q$ if all the partial sums $S_k(t):=\sum_{i=1}^k \lambda_i(t)$ for $k=1\dots n$  are nondecreasing. That is, for all $s,t \in [0,1]$ such that  $s\geq t$: $S_k(s) \geq S_k(t)$, or said otherwise this means that   a continuous convex majorization holds:
\[
 sp+(1-s)q \succeq  tp+(1-t)q .
\]
In particular strong majorization implies majorization. We will denote this strong majorization by $p \succeq^s q$.
\end{definition}

Now, the main result from Section~3 reads as follows:

\begin{reptheorem}{NSCM}
For polynomials $p$ and $q$ with a common interlacer,
\[
p \succeq^s q \iff \sum_{i=1}^{i=k} \frac{q[\lambda_i]}{p'[\lambda_i]} \geq 0 \text{  } \forall k<n \text{ and } \sum_{i=1}^n\lambda_i= \sum_{i=1}^n\mu_i.
\]
\end{reptheorem}
It should be noted that such a necessary and sufficient characterization of strong majorization is extremely easy to check numerically: we just have to compute the partial sums of residues. Therefore, some nontrivial information on relative positioning of roots can be deduced from easy inspection of the quotient of the two polynomials. \\
In Section~4, we detail the difference between majorization and strong majorization. Namely, we prove the following.

\begin{reptheorem}{DiffMaj}[Simple majorization without strong majorization]
Assume $p$ and $q$  admit a common interlacer, have distinct roots and  $p \succeq q$ . \\
 If   $\sum_{i=1}^k \lambda_i = \sum_{i=1}^k \mu_i$ for  $1<k<n$, then there exists a partial sum of residues\\
 $ \sum_{i=1}^{i=k_0} \frac{q[\lambda_i]}{p'[\lambda_i]} <0$, for $k_0 \leq k$. So that in this case, there is no strong majorization.
\end{reptheorem}

We then deduce  a more accurate characterization of strong majorization. If the partial sums of roots $S_k(t)$ are all nondecreasing for $k<n$, then they are in fact necessarily increasing.

\begin{repcor}{strictmon}[Strict monotonicity]
Assume $p$ and $q$ admit a common interlacer and have distinct roots, as well as $\sum_{i}^n\lambda_i= \sum_{i}^n\mu_i$.Then:
\[
   p \succeq^s q  \iff \text{ }\forall k<n,  S_k(t)=\sum_{i=1}^k\lambda_i(t) \text{ is increasing on } [0,1]. 
\]\end{repcor}

\subsection{Classical tools from majorization theory}
Let us recall a few basic facts about majorization that will be useful in the subsequent proofs.
\begin{lemma}[Robin-Hood characterization, also called Dalton in \cite{BMAJ}] \label{Robin} Assume $a\succeq b$. 
From $a$ we can obtain $b$ by a finite sequence of operations where we replace two elements $a_i$ and $a_j$ such that $a_j<a_i$  ($j<i$) with $a_i-\epsilon$ and $a_j+\epsilon$ respectively for an $ \epsilon \in ]0,a_i-a_j[$. 
\end{lemma}

We will need the notion of Schur convexity to analyze transformations of vectors exhibiting monotonicity with respect to majorization.

\begin{definition}[Definition A.1 in \cite{BMAJ}]\label{schurconvex}
A function $\phi$ defined on an open set $\mathcal{D} \subset \mathbb{R}^l$ is said to be Schur convex on $\mathcal{D}$ if for $x,y \in \mathcal{D}$:
\begin{align*}
 y &\succeq x  &  \implies&& \phi(y) \geq \phi(x). 
\end{align*}
If in addition  $\phi(y) > \phi(x)$ whenever the vectors are distinct (up to permutation, but we assume that they are ordered by decreasing order), then we say that $\phi$ is strictly Schur convex.  
\end{definition}

We can characterize Schur convexity with respect to partial derivatives.

\begin{theorem}[Global Schur-Ostrowski criterium, Theorem A.4 in \cite{BMAJ}]\label{ostroglobal}
A $C^1$ function $\phi(x_1,\dots,x_l)$ defined on an open set $\mathcal{D}^l \subset \mathbb{R}^l$ is strictly Schur convex on $\mathcal{D}^l$ if and only if $\phi$ is symmetric in its $l$ variables and for all $1\leq i\neq j \leq l$, $x_i \neq x_j$:
\[
(x_i-x_j) \Big(\partial_{x_i}{\phi(x)}-\partial_{x_j}{\phi(x)}\Big) > 0. 
\]
\end{theorem}

We will also use a restrictive version of Schur concavity using only two coordinates. 
 
\begin{lemma}[Local Schur-Ostrowski criterium, Lemma A.2 and Theorem A.3 in \cite{BMAJ}]\label{ostro}
 Consider a function $\phi$ on $n$ variables within an open subset $\mathcal{D}$ of $ \mathbb{R}^n$ that is $C^1$. For $\epsilon>0$ and  $x=(x_1,\dots,x_n) $ , define $x_{\epsilon}^{i,j}:= (x_1,\dots, x_i-\epsilon, \dots, x_j+\epsilon,\dots, x_n)$ where $i<j$  and where we performed a  Robin-Hood operation so that: $  x \succeq x_{\epsilon}^{i,j}$. Consider $\mathcal{D}_\epsilon^{i,j}:= \{ x \in \mathcal{D}, \text{ such that } x_\epsilon^{i,j} \in \mathcal{D}\}$. Then we have:
\begin{align*}
\partial_{x_i}\phi(y)  &< \partial_{x_j}\phi(y) \text{  } \forall y \in \mathcal{D}\implies&& \phi(x_{\epsilon}^{i,j} )>\phi(x)  \text{  } \forall x \in \mathcal{D}_\epsilon^{i,j},
\end{align*}
and:
\begin{align*}
\partial_{x_i}\phi(y)  &> \partial_{x_j}\phi(y) \text{  } \forall y \in \mathcal{D}\implies&& \phi(x_{\epsilon}^{i,j} )<\phi(x)  \text{  } \forall x \in \mathcal{D}_\epsilon^{i,j}.
\end{align*}
\end{lemma}

\section{Necessary condition for majorization}

We will need a monotonicity lemma for weighted partial sums of roots. 
\begin{lemma} \label{lemmabis}
Let $0<r_1<r_2\dots<r_k$ and $\delta_1,\dots,\delta_k$ such that for an index $s_0$, $s_0<k$, $\sum_{i=1}^{s_0} \delta_i < 0$ and for all  $s<k$, $\sum_{i=1}^s \delta_i  \leq  0$. Then $\sum_{i=1}^k \delta_i < \sum_{i=1}^k \delta_i \frac{r_i}{r_k}  $. In particular, if in addition 
$\sum_{i=1}^k \delta_i r_i \leq 0$, then $\sum_{i=1}^k \delta_i < 0$.
\end{lemma}
\begin{proof}
Consider the following decomposition:
 \[
 \sum_{i=1}^k (r_k-r_i)\delta_i =\sum_{i=1}^{k-1} \sum_{l=i}^{k-1}(r_{l+1}- r_{l})\delta_i   =  \sum_{l=1}^{k-1} (r_{l+1}- r_{l}) \sum_{i=1}^{l} \delta_i.
 \]
As $(r_{l+1}- r_{l})>0$ and $\sum_{i=1}^{l} \delta_i \leq 0$ for all $l<k$, and one of them at least is negative, we obtain  $\sum_{i=1}^k (r_k-r_i)\delta_i< 0$. 
\end{proof}

Let us recall that we denote the vector of roots of $q$ by $\mu$ and the vector of roots of $p$ by $\lambda$, ordered in decreasing order. 
\begin{theorem}[Necessary condition] \label{NCM}
If $p$ and $q$ have a common interlacer and distinct roots, then
\[
p \succeq q \implies \sum_{i=1}^{i=k} \frac{p[\mu_i]}{q'[\mu_i]} < 0  \text{ } \forall k< n .
\]
\end{theorem}

\begin{proof}
 All along the vector of roots $\mu$ will be fixed. Denote by :
\[
f_k^{\mu}(\lambda_1,\dots.,\lambda_n)= \sum_{i=1}^{i=k} \frac{p[\mu_i]}{q'[\mu_i]} =  \sum_{i=1}^{i=k} \frac{\prod_{l=1} ^n(\mu_i - \lambda_l)}{\prod_{l \neq i, l=1}^{n}(\mu_i-\mu_l)} =   \sum_{i=1}^{i=k}\delta_i^{\mu}(\lambda)
\]
where for all $i=1\dots n$, 
\begin{align*}
\delta_i^{\mu}(\lambda)&:=  \frac{\prod_{l=1} ^n(\mu_i - \lambda_l)}{\prod_{l \neq i, l=1}^{n}(\mu_i-\mu_l)}.
\end{align*}
Notice that by the common interlacer assumption, $(\mu_i- \lambda_l)(\mu_i-\mu_l)>0$ for all $l,i=1 \dots n$(they are ordered by pairs into disjoint intervals, see Definition~\ref{interl}). \\
Therefore $\sign(\delta_i^{\mu}(\lambda))= \sign(\mu_i - \lambda_i)$.

We prove the inequality by induction on $k$. The case $k=1$ is more or less straightforward as by majorization $\lambda_1 \geq  \mu_1$, and the sign of the fraction $\delta_1^{\mu}(\lambda)$ is the sign of $(\mu_1-\lambda_1)$: $\delta_1^{\mu}(\lambda)<0$.

Fix $k$. Assume it is true for all $q<k$. We will operate transformations on $\lambda$ that preserve the common interlacing and the majorization properties. 
Now, an easy computation leads, for $1\leq i_1\leq n$, to
\[
\frac{\partial f_k^{\mu}}{\partial \lambda_{i_1}}=   -\sum_{i=1}^k \frac{\delta_i^{\mu}(\lambda)}{\mu_i-\lambda_{i_1}},
\]
and then for all $1\leq i_1,i_2\leq n$:
\[
\frac{\partial f_k^{\mu}}{\partial \lambda_{i_1}} -\frac{\partial f_k^{\mu}}{\partial \lambda_{i_2}} = (\lambda_{i_2} -\lambda_{i_1})  \sum_{i=1}^{i=k} \frac{\delta_i^{\mu}(\lambda)}{(\mu_i-\lambda_{i_1})(\mu_i-\lambda_{i_2})}.
\]
 As long as we can find two indices,$i_1$,$i_2$, with  $i_2>i_1>k$  (therefore by assumption $\lambda_{i_1}> \lambda_{i_2}$) and:
\begin{align}\label{cond}
 \lambda_{i_1}&>\mu_{i_1} &\text{and}&&  \lambda_{i_2}&<\mu_{i_2},
 \end{align}
 do the following:  for all $1\leq i\leq k$, $r_i:= \frac{1}{(\mu_i-\lambda_{i_1})(\mu_i-\lambda_{i_2})} >0$ and $r_1<r_2\dots<r_s$  due to the assumption that the indices $i_1,i_2>k$. 
Then there is a dichotomy.\\
\begin{enumerate}
 \item If $\sum_{i=1}^k \delta_i^{\mu}(\lambda) r_i \leq 0$ we can use Lemma \ref{lemmabis} to conclude directly that $f_k(\lambda_1,\dots.,\lambda_n)= \sum_{i=1}^k \delta_i^{\mu}(\lambda) \leq 0$ as by induction  $\sum_{i=1}^s\delta_i^{\mu}(\lambda) \leq 0$ for all $s<k$ and $\delta_1^{\mu}(\lambda)<0$ .\\
 
\item  If $\sum_{i=1}^k\delta_i^{\mu}(\lambda) r_i > 0$,  then consider vectors $x \in \mathbb{R}^n$ that have all coordinates constant equal to $x_j= \lambda_j$ except for $j=i_1,i_2$. Define the set of admissible coordinates as $\mathcal{D}_{i_1,i_2}= \{x,  \text{where  }  x_{i_1}>\mu_{i_1},  x_{i_2}<\mu_{i_2} \text{ and }\frac{\partial f_k^{\mu}}{\partial x_{i_1}}(x) <\frac{\partial f_k^{\mu}}{\partial x{i_2}}(x)\}$. Notice that $\mathcal{D}_{i_1,i_2}$ is open when restricted to the subspace of the two moving coordinates space. As:
 \begin{align*}
 \lambda_{i_1}&>\mu_{i_1},     & \lambda_{i_2}&<\mu_{i_2} &\text{and} &&\frac{\partial f_k^{\mu}}{\partial \lambda_{i_1}}(\lambda)<\frac{\partial f_k^{\mu}}{\partial \lambda_{i_2}}(\lambda)
 \end{align*}
 by assumption, then $\lambda \in \mathcal{D}_{i_1,i_2}$. 
 Consider $ \lambda^{\epsilon}= (\lambda_1,\dots, \lambda_{i_1}-\epsilon, \dots, \lambda_{i_2}+\epsilon,\dots, \lambda_n)$, and $\epsilon_{max}= \sup\{\epsilon, \text{ such that }  \lambda^{\epsilon}\in  \mathcal{D}_{i_1,i_2}\} $.  So we squeeze the vector of $\lambda$ to make it closer to $\mu$ by a Robin Hood operation (see Definition~\ref{Robin}). According to Lemma~\ref{ostro}, we have:
 \[
 f_k^{\mu}(\lambda^{\epsilon_{max}})> f_k^{\mu}(\lambda).
 \]

  By continuity at the extreme point of $\mathcal{D}_{i_1,i_2}$ , either $\frac{\partial f_k^{\mu}}{\partial \lambda_{i_1}}(\lambda^{\epsilon_{max}})=\frac{\partial f_k^{\mu}}{\partial \lambda_{i_2}}(\lambda^{\epsilon_{max}})$ and in this  case, we are brought back to case 1) and we leave the loop, the inequality is proven. Or if $\lambda^{\epsilon_{max}}_{i_1}=\mu_{i_1}$ or $ \lambda^{\epsilon_{max}}_{i_2}=\mu_{i_2} $ we replace  $\lambda$ by $ \lambda^{\epsilon_{max}}$ and we start again the dichotomy with new indices $i_1$ and $i_2$. 
 \end{enumerate}
Call $\lambda^{end}$ the transformed vector we get at the end. Note that as at all steps $\lambda^{\epsilon_{max}} \succeq \mu$ by construction, it his therefore true for the final vector: $\lambda^{end} \succeq \mu$. Also the roots stay in disjoint intervals all along so the property of having a common interlacer is kept as an invariant. The negation of inequalities from Equation~\ref{cond} for all indices larger than $k$ means that we can't squeeze $\lambda^{end}$ more to be closer to $\mu$ and given that  $\lambda^{end} \succeq \mu$, we necessarily have: $\mu_r \geq \lambda_r^{end}$ for all $r \in \left[|k+1,n\right|]$ and this  leads to $\delta_r^{\mu}(\lambda^{end}) \geq 0$ on this range. \\
Finally, as $ \sum_{i=1}^n \delta_i^{\mu}(\lambda^{end})= \sum_{i=1}^n \mu_i -  \sum_{i=1}^n \lambda_i^{end} =0 $ by the majorization property, we deduce that 
 $f_k^{\mu}(\lambda^{end})= -  \sum_{r=k+1}^n\delta_i^{\mu}(\lambda^{end}) \leq 0$. As $f_k^{\mu}(\lambda)< f_k^{\mu}(\lambda^{end})$ by construction, we get the desired inequality(note that if $k=n$, such a process would not be possible). 

\end{proof}

\section{A sufficient and necessary condition for strong majorization }
In this section, we prove results related to strong majorization that also lead to a sufficient characterization in terms of residues for simple majorization. Denote by $\lambda_i(t)$ the roots of  $p_t:= t p+(1-t)q$. We have $\lambda_i(0)=\mu_i$ and $\lambda_i(1)=\lambda_i$.  Also, for $k\leq n$, recall that $S_k(t)= \sum_{i=1}^k \lambda_i(t)$.
\begin{lemma}[Sufficient condition for strong majorization]\label{strictineq} Assume that $p$ and $q$ have a common interlacer and that  $\sum_{i=1} \lambda_i=\sum_{i=1}^n \mu_i$.
Then: 
\[
 \forall  k<n, \sum_{i=1}^{i=k} \frac{q[\lambda_i]}{p'[\lambda_i]} > 0 \implies  \forall k<n, \forall t\in[0,1], \frac{dS_k(t)}{dt}>0  \implies   p \succeq^s q.
\]
\end{lemma}
\begin{proof}
Let us look at the equations of evolution of the partial sums of roots with respect to $t$. 
The goal is to prove that they are all increasing.

  Let's differentiate the equalities:
\begin{equation}\label{zeroeq}
\big( t p+(1-t)q\big) [\lambda_i(t)] =p_{t} [\lambda_i(t)] =0,
\end{equation}
 we get for $0< t<1$:
\begin{equation}\label{derivroots}
\frac{d \lambda_i(t)}{dt}= \frac{(q-p)}{p'_t} [ \lambda_i(t)]= \frac{1}{t}\frac{q}{p'_t} [ \lambda_i(t)] =  \frac{-1}{1-t}\frac{p}{p'_t} [ \lambda_i(t)].
\end{equation}
We then have:
 \[
 \frac{dS_k(t)}{dt}=\sum_{i=1}^k\frac{d \lambda_i(t)}{dt} = \frac{1}{t} \sum_{i=1}^k \frac{q}{p'_t} [ \lambda_i(t)].
 \]
 By assumption, we have, for $k<n$:
 \[
  \frac{dS_k(t)}{dt}|_{t=1}=\sum_{i=1}^k \frac{q}{p'} [\lambda_i]>0.
 \]
 We want to show that $\frac{dS_k(t)}{dt}\geq 0$  for all $t \in [0,1]$. By continuity of the functions involved with respect to $t$ (roots depend continuously on the coefficients and therefore on $t$, see Equation~\ref{zeroeq}) , for $t$ close to $1$, and for all $k<n$: 
\[
  \frac{dS_k(t)}{dt} >0.
\]
Now assume by contradiction that sums  $S_k(t)$ are not all increasing, then one  $\frac{dS_k(t)}{dt}$ has to become negative. that is there exist  $k_0<n$  and  $0\leq t_0<1$, such that $\frac{dS_k(t)}{dt}(t_0)=0$(which has to happen by continuity and using the intermediate value theorem). Assume that $t_0:= \max \{ 0<t<1, \text{ such that } \frac{dS_k(t)}{dt}(t_0)=0 \text{ for any } k<n\} $. As we also have, using Equation~\ref{derivroots}:
\[
  \frac{dS_k(t)}{dt}= \frac{-1}{1-t}\sum_{i=1}^k\frac{p}{p'_t} [ \lambda_i(t)],
\]
 we eventually get:
  \begin{align*}
 \sum_{i=1}^{k_0} \frac{p}{p'_{t_0}} [ \lambda_i(t_0)] &=0    &\text{and}&&   \sum_{i=1}^{k} \frac{p}{p'_{t_0}} [ \lambda_i(t_0)]  \leq 0  \text{ for }  k\neq k_0.
 \end{align*}
  Also  notice that as $\frac{dS_k(t)}{dt}=\sum_{i=1}^k \frac{d \lambda_i(t)}{dt}> 0$ for all $k<n$ and all $t \in]t_0,1]$, and by assumption $S_n(t)=S_n(1)$,  then there is continuous majorization between $p_{t_0}$ and $p$ (partial sums of roots are all increasing), and in particular $p \succeq p_{t_0}$. \\
  As  the largest roots of $p$ and $q$ are distinct, similarly  the largest roots of $p$ and $p_{t_0}$ are distinct. Indeed, $p_{t_0}=t_0p+(1-t_0)q$ so for $t_0<1$, if the largest root of $p_{t_0}$ was the largest of $p$, it would also be a root of $q$ which is impossible (by the disposition into intervals of the roots, it would also be the largest root of $q$). \\
 The polynomials $p$ and $p_{t_0}$ also have a common interlacer. Indeed if a polynomial $r$ interlaces both $p$ and $q$ then it also interlaces both $p_{t_0}=t_0p+(1-t_0)q$  and $p$ (see Lemma~\ref{commoninterl}).\\
  We can see that this situation is impossible following Theorem~\ref{NCM} applied to $p$ and $p_{t_0}$. \\
 We conclude that there is a contradiction. Finally, for all $k<n$ and all $t \in [0,1]$, $\frac{dS_k(t)}{dt}>0$. Therefore we get the desired strong majorization. Note that we couldn't have done this at $t_0=1$, because of the factor $\frac{1}{1-t}$; that's why we need strict inequalities at $t=1$ to get rid of this singularity and to be able to use both equalities with $q$ or $p$ on top of the denominator (and go from one to the other). Also note that for $t=0$, we retrieve the inequalities from
\end{proof}
To extend this result to large inequalities, we will need a variant of Lemma~\ref{lemmabis} where inequalities are reversed.
\begin{lemma} \label{lemmabispos}
Let $r_1>r_2\dots>r_k>0$ and $\delta_1,\dots,\delta_k$ such that for an index $s_0$, $s_0<k$, $\sum_{i=1}^{s_0} \delta_i > 0$ and for all  $s \leq k$, $\sum_{i=1}^s \delta_i  \geq  0$. Then $\sum_{i=1}^k \delta_i r_i > \sum_{i=1}^k \delta_i r_k   $. In particular, $\sum_{i=1}^k \delta_ir_i >0$.
\end{lemma}
\begin{proof}
Consider the following decomposition:
 \[
 \sum_{i=1}^k (r_i-r_k)\delta_i =\sum_{i=1}^{k-1} \sum_{l=i}^{k-1}(r_{l}- r_{l+1})\delta_i   =  \sum_{l=1}^{k-1} (r_{l}- r_{l+1}) \sum_{i=1}^{l} \delta_i.
 \]
As  $\forall l<k$, $(r_{l+1}- r_{l})>0$ and $\sum_{i=1}^{l} \delta_i \geq 0$ f  and one of them is positive, we obtain  $\sum_{i=1}^k (r_i-r_k)\delta_i > 0$. 

\end{proof}

\begin{cor} [extension to large inequalities]\label{suffcor} Assume $p$ and $q$ share a common interlacer and that $ \sum_{i=1}^{n}\lambda_i= \sum_{i=1}^{n}\mu_i $.

Then:
\[
\forall k<n, \sum_{i=1}^{i=k} \frac{q[\lambda_i]}{p'[\lambda_i]} \geq 0 \implies p \succeq^s q.
\]
\end{cor}
\begin{proof}
If the partial sums $\sum_{i=1}^{i=k} \frac{q[\lambda_i]}{p'[\lambda_i]}$ are all  positive for $1<k<n$ then we are brought back to the case of Lemma~\ref{strictineq}. So we can assume that at least one is zero. The goal is to show that we can approximate $q$ by some perturbed polynomial  $q_{\epsilon}$ for which we can apply the previous lemma and then make the $\epsilon$ perturbation go to zero. 
Denote by:
\[
 k_0= \min \{k, 1<k<n \text{ such that } \sum_{i=1}^{i=k} \frac{q[\lambda_i]}{p'[\lambda_i]} = 0\} .
 \] The case $k=1$ is impossible  because it would mean that the largest roots are the same (which we excluded). Similarly we necessarily have  $\frac{q[\lambda_i]}{p'[\lambda_i]} \neq 0$, for all $i \leq n$. So It means that  $\frac{q[\lambda_{k_0}]}{p'[\lambda_{k_0}]}<0$, and in particular, given that $\sign(\frac{q[\lambda_{k_0}]}{p'[\lambda_{k_0}]})= \sign(\lambda_{k_0}-\mu_{k_0})$, that: $\mu_{k_0}>\lambda_{k_0}$. We can't have $k_0=n-1$  as $ \sum_{i=1}^{i=n} \frac{q[\lambda_i]}{p'[\lambda_i]} = 0$, and it would lead to $\lambda_n=\mu_n$. So we necessarily have: $\frac{q[\lambda_{k_0+1}]}{p'[\lambda_{k_0+1}]}> 0$, as $\sum_{i=1}^{i=k_0+1} \frac{q[\lambda_i]}{p'[\lambda_i]} \geq 0$.

Denote by :
\begin{align*}
f_k^{\lambda}(\mu_1,\dots.,\mu_n)&:= \sum_{i=1}^{i=k} \frac{q[\lambda_i]}{p'[\lambda_i]}   &\text{and}&&    g_k^{\lambda}(\mu_1,\dots.,\mu_n)&:= \sum_{i=k+1}^{i=n} \frac{q[\lambda_i]}{p'[\lambda_i]}.
\end{align*}

Now, for general vectors $\lambda$ and $\mu$ (let us not assume that there is majorization):
 \begin{align*}
 f_k^{\lambda}(\mu)+  g_k^{\lambda}(\mu)= \sum_{i=1}^n (\lambda_i - \mu_i) \implies  \frac{\partial (f_k^{\lambda} +g_k^{\lambda})}{{\partial \mu_l}} = -1   \implies  \frac{\partial( f_k^{\lambda} +g_k^{\lambda})}{{\partial \mu_{l_1}}} - \frac{\partial (f_k^{\lambda} +g_k^{\lambda})}{{\partial \mu_{l_2}}}   =0,
 \end{align*}
for all indices $l_1$ and $l_2$, or put otherwise:
\[
\frac{\partial f_k^{\lambda}}{{\partial \mu_{l_1}}} - \frac{\partial f_k^{\lambda}}{{\partial \mu_{l_2}}}  = \frac{\partial g_k^{\lambda}}{{\partial \mu_{l_2}}} - \frac{\partial g_k^{\lambda}}{{\partial \mu_{l_1}}} .
\]
 We have $\mu_{1}< \lambda_{1}$ by assumption, and
\[
\frac{\partial g_k^{\lambda}}{\partial \mu_{1}} -\frac{\partial g_k^{\lambda}}{\partial \mu_{k_0}} = (\mu_{k_0} -\mu_{1})  \sum_{i=k+1}^{n}  \frac{q[\lambda_i]}{p'[\lambda_i]}  \frac{1}{(\lambda_i-\mu_{1})(\lambda_i-\mu_{k_0})}.
\]
We notice that for $i > k_0$, if we denote by $r_i: = \frac{1}{(\lambda_i-\mu_{1})(\lambda_i-\mu_{k_0})}$, then $ r_{k_0+1}> r_{k_0+2} >\dots.>r_n>0$ . So as by assumption $f_r^{\lambda}(\mu) \geq 0$ for all $r$ and  $f_{k_0}^{\lambda}(\mu)= \sum_{i=1}^{k_0} \frac{q[\lambda_i]}{p'[\lambda_i]} = 0$ then it also means that  $\sum_{i=k_0+1}^{l} \frac{q[\lambda_i]}{p'[\lambda_i]} \geq 0$ for $l>k_0$. Using Lemma~\ref{lemmabispos}  and  : $\frac{q[\lambda_{k_0+1}]}{p'[\lambda_{k_0+1}]}> 0$, we get that:   $\sum_{i=k_0+1}^{n}  \frac{q[\lambda_i]}{p'[\lambda_i]}  \frac{1}{(\lambda_i-\mu_{1})(\lambda_i-\mu_{k_0})} >0$. We conclude (as $(\mu_{k_0} -\mu_{1})  <0$) that:
\begin{align*}
\frac{\partial g_{k_0}^{\lambda}}{\partial \mu_{1}}(\mu) -\frac{\partial g_{k_0}^{\lambda}}{\partial \mu_{k_0}}(\mu) &<0 & \text{and}&& \frac{\partial f_{k_0}^{\lambda}}{\partial \mu_{1}}(\mu) -\frac{\partial f_{k_0}^{\lambda}}{\partial \mu_{k_0}}(\mu) &>0.
\end{align*}
Exactly the same way, if $k$ is an other index (larger)  such that $\sum_{i=1}^{i=k} \frac{q[\lambda_i]}{p'[\lambda_i]} = 0$ and which is not equal to $n$, we have that for $i > k$, if we denote by $r_i: = \frac{1}{(\lambda_i-\mu_{1})(\lambda_i-\mu_{k_0})}$, then $ r_{k+1}> r_{k+2} >\dots.>r_n>0$ . So as by assumption $f_r^{\lambda}(\mu) \geq 0$ for all $r$ and  $f_{k}^{\lambda}(\mu)=\sum_{i=1}^{k} \frac{q[\lambda_i]}{p'[\lambda_i]} = 0$ then it also means that  $\sum_{i=k+1}^{l} \frac{q[\lambda_i]}{p'[\lambda_i]} \geq 0$ for $l>k$. Using Lemma~\ref{lemmabispos}  and  : $\frac{q[\lambda_{k+1}]}{p'[\lambda_{k+1}]}> 0$, we get that:   $\sum_{i=k+1}^{n}  \frac{q[\lambda_i]}{p'[\lambda_i]}  \frac{1}{(\lambda_i-\mu_{1})(\lambda_i-\mu_{k})} >0$. We conclude (as $(\mu_{k_0} -\mu_{1})  <0$) that:
\begin{align*}
\frac{\partial g_{k}^{\lambda}}{\partial \mu_{1}}(\mu) -\frac{\partial g_{k}^{\lambda}}{\partial \mu_{k_0}}(\mu)& <0 &\text{and}&& \frac{\partial f_{k}^{\lambda}}{\partial \mu_{1}}(\mu) -\frac{\partial f_{k}^{\lambda}}{\partial \mu_{k_0}}(\mu)& >0.
\end{align*}
We do a small Robin Hood transfer of weight from $\mu_1$ to $\mu_{k_0}$, that is we define:
 \begin{align*}
 \mu_1^{\epsilon}&= \mu_1- \epsilon, & \mu_{k_0}^{\epsilon}&= \mu_{k_0}+ \epsilon  &\text{and}&&    \mu_{i}^{\epsilon}&= \mu_i \text{ }\forall i \neq 1,k_0.
 \end{align*}
 We denote by $q_{\epsilon}$ the polynomial whose roots are $\mu_i^{\epsilon}$.  Then for $\epsilon$ small enough, $q_{\epsilon}$  and $p$ will still have a common interlacer, and 
by continuity, if $f_k(\mu)>0$ then $f_k(\mu_{\epsilon})>0$ also. Additionally, if $f_k(\mu)=0$ then using the above inequality $\frac{\partial f_{k}^{\lambda}}{\partial \mu_{1}}(\mu)>\frac{\partial f_{k}^{\lambda}}{\partial \mu_{k_0}}(\mu)$ and Lemma~\ref{ostro} with $i=1$ and $j=k_0$ on a small domain $D$ around $\mu$, we get $f_k(\mu_{\epsilon})> f_k(\mu)=0$. Then we have strict inequalities for $q_{\epsilon}$ and we can apply the previous Lemma~\ref{suffcor} and if we denote by   $\lambda_i(t,\epsilon)$ the roots of $tp+(1-t)q_{\epsilon}$, we get for $0 \leq t<s \leq 1$ and all $k<n$:

\[
 \sum_{i=1}^k  \lambda_i(t,\epsilon)<\sum_{i=1}^k  \lambda_i(s,\epsilon).
\]
  Using  the fact that the coefficients of $tp+(1-t)q_{\epsilon}$  are $\epsilon$ close to the coefficients of  $tp+(1-t)q$ and then using continuity of the roots with respect to the coefficients, we get that \\
   $ \lambda_i(t,\epsilon) \longrightarrow_{\epsilon \to 0}  \lambda_i(t)$, as well as  $ \lambda_i(s,\epsilon) \longrightarrow_{\epsilon \to 0}  \lambda_i(s)$ and finally:
   \[
    \sum_{i=1}^k  \lambda_i(t)\leq \sum_{i=1}^k  \lambda_i(s).
   \]
\end{proof}

\begin{theorem}\label{NSCM}Assume again $p$ and $q$ share a common interlacer and  $ \sum_{i=1}^{n}\lambda_i= \sum_{i=1}^{n}\mu_i $.
We have the equivalence:

\[
\text{for all $k< n,$}  \sum_{i=1}^{i=k} \frac{q[\lambda_i]}{p'[\lambda_i]} \geq 0     \iff   p \succeq^s q
\]
\end{theorem}

\begin{proof}
It is sufficient according to Corollary~\ref{suffcor}. The only remaining part is the necessity. So assume strong majorization holds. We have:
\[
\frac{dS_k(t)}{dt}|_{t=1}=\sum_{i=1}^k  \frac{d  \lambda_i(1)}{dt}=  \sum_{i=1}^k \frac{q[\lambda_i]}{p'[\lambda_i]}.  
\]
In particular, this quantity has to be nonnegative my monotonicity in a neighborhood of $1$, which proves what we want. 
\end{proof}

\section{Strong majorization versus simple majorization}
Now let us investigate if it is possible that some partial sums are negative (so absence of strong majorization) despite majorization. It will show that strong majorization is indeed a stronger notion than simple majorization (as majorization can take place without strong majorization). 

\begin{theorem}\label{DiffMaj}
Assume $p \succeq q$ admit a common interlacer and have distinct roots.  If there is intermediate equality for partial sums of roots, that is  $\sum_{i=1}^k \lambda_i = \sum_{i=1}^k \mu_i$ for  $k<n$ (and of course $k>1$), then there exists a partial sum of residues $ \sum_{i=1}^{i=k_0} \frac{q[\lambda_i]}{p'[\lambda_i]} <0$, for $k_0 \leq k$. So in this case, there is no strong majorization.
\end{theorem}

\begin{proof}
Assume there exists $k<n$ such that  $\sum_{i=1}^k \lambda_i = \sum_{i=1}^k \mu_i$. Write for $i\leq k$:
\[
\frac{q[\lambda_i]}{p'[\lambda_i]} =\frac{\prod_{l=1} ^k(\lambda_i - \mu_l)}{\prod_{l \neq i, l=1}^{k}(\lambda_i-\lambda_l)} \prod_{j=k+1}^n \frac{(\lambda_i- \mu_j)}{(\lambda_i-\lambda_j)}=   \Delta^k(\lambda_i)Q^k(\lambda_i),
\]
where:
\begin{align*}
\Delta^k(x)&:= \frac{\prod_{l=1} ^k(x - \mu_l)}{\prod_{l \neq i, l=1}^{k}(x-\lambda_l)} & \text{and} &&Q^k(x):= \prod_{j=k+1}^n \frac{(x- \mu_j)}{(x-\lambda_j)}.
\end{align*}

\begin{lemma}
We have $\frac{d Q^k(x)}{dx} <0$ for $x \in  \left[\lambda_k,\lambda_1 \right]$. 
\end{lemma}
\begin{proof}
Write  $Q_{\mu}^k(x):= \prod_{j=k+1}^n(x- \mu_j) $ and $Q_{\lambda}^k(x):= \prod_{j=k+1}^n(x- \lambda_j)$. 
As $Q_{\mu}^k$ and $Q_{\lambda}^k$ are positive for $x \in \left[\lambda_k,\lambda_1 \right]$,  
\[
 \sign(\frac{d Q^k(x)}{dx})= \sign( \frac{Q_{\mu}'^k(x)}{Q_{\mu}^k(x)} -   \frac{Q_{\lambda}'^k(x)}{Q_{\lambda}^k(x)} ) = \sign(\sum_{j=k+1}^n\frac{1}{x-\mu_j}- \sum_{j=k+1}^n \frac{1}{x-\lambda_j} ).
 \]
 Now, call $ h_x( \nu_{k+1},\dots,\nu_n):= \sum_{j=k+1}^n \frac{1}{x-\nu_j}$. We have, for $k<i,j \leq n$,
 \[
 \frac{\partial h_x}{\partial \nu_i} -  \frac{\partial h_x}{\partial \nu_j} = \frac{1}{(x-\nu_i)^2} -  \frac{1}{(x-\nu_j)^2}= \frac{(\nu_i-\nu_j)(2x-\nu_i-\nu_j) }{[(x-\nu_i)(x-\nu_j)]^2}.
 \]
 So that, for $x \in \left[\lambda_k,\lambda_1 \right]$, and  $\nu_i \neq \nu_j< \lambda_k$, $2x >\nu_i+\nu_j$,
 
 \begin{equation}\label{partialderiv}
 (\nu_i -\nu_j) \Big[\frac{\partial h_x}{\partial \nu_i} -  \frac{\partial h_x}{\partial \nu_j} \Big]= \frac{(\nu_i-\nu_j)^2(2x-\nu_i-\nu_j) }{[(x-\nu_i)(x-\nu_j)]^2}>0.
 \end{equation}
 Define $\mathcal{D}:= \{\nu, \nu<\lambda_k\}$. Then as  $ h_x( \nu_{k+1},\dots,\nu_n)$ is symmetric and verifies Equations~\ref{partialderiv} on $\mathcal{D}^{n-k}$, Theorem~\ref{ostroglobal} tells us that $h_x$ is Schur convex on $\mathcal{D}^{n-k}$ .
 Here comes the crucial part. As  $\sum_{j=1}^k \lambda_j = \sum_{j=1}^k \mu_j$, we have for all $i>0$ such that $i+k \leq n$:  $\sum_{j=k+1}^{k+i} \lambda_j \geq \sum_{j=k+1}^{k+i} \mu_i$ by the fact that $p \succeq q$, which leads to $(\lambda_{k+1},\dots.,\lambda_n) \succeq (\mu_{k+1},\dots.,\mu_n)$. This majorization of the vector of roots starting at the index $k+1$ plus the strict Schur convexity of $h_x$ leads (see Definition~\ref{schurconvex}) to:  
  \[
  h_x( \lambda_{k+1},\dots,\lambda_n) > h_x( \mu_{k+1},\dots,\mu_n),
  \]
  which in turn implies for $x \in \left[\lambda_k,\lambda_1 \right]$; $ \frac{d Q^k(x)}{dx} <0$.
 \end{proof}
 In particular we get: $Q^k(\lambda_k)>Q^k(\lambda_{k-1})\dots>Q^k(\lambda_1)>0$.
Now assume by way of contradiction that for all $1\leq j \leq k$, $ \Sigma_j:= \sum_{i=1}^j \frac{q}{p'} [\lambda_i] =\sum_{i=1}^j \Delta^k(\lambda_i)Q^k(\lambda_i) \geq 0$ (in particular $\Sigma_j$ are positive linear combinations of the $\Delta^k(\lambda_i)$)  .
\begin{lemma}
 We can express $\sum_{j=1}^k \Delta^k(\lambda_j)$ as a positive combination of the $\Sigma_j$, that is there exist a sequence $\alpha_j>0$ such that: 
\[
\sum_{j=1}^k\Delta^k(\lambda_j)= \sum_{j=1}^k \alpha_j \Sigma_j.
\]
\end{lemma}
\begin{proof}
Let us recursively define the coefficients, using the identity:
\[
\sum_{j=1}^k\Delta^k(\lambda_j)=  \sum_{j=1}^k\Delta^k(\lambda_j)\sum_{i=j}^k \alpha_i Q^k(\lambda_j).
\]

 We want to get for all $j\leq k$:
 \[
 1= Q^k(\lambda_j)\sum_{i=j}^k \alpha_i.
 \] 
 Let us prove that these equalities determine the $\alpha_j$ sequence uniquely. We have $\alpha_k= \frac{1}{Q^k(\lambda_k)}>0$. Now we proceed by induction. Assume that for $j>j_0$,  $\alpha_j>0$ is such that: $\alpha_j Q^k(\lambda_{j}) + Q^k(\lambda_{j})\sum_{i=j+1}^k  \alpha_i$ =1 is uniquely defined. We are looking for $\alpha_{j_0}>0$ such that : $ \big(\sum_{j=j_0}^k \alpha_j \big) Q^k(\lambda_{j_0}) =1$.
But
\[
1- (\sum_{j=j_0+1}^k \alpha_j \big) Q^k(\lambda_{j_0}) = 1-  \Big[\sum_{j=j_0+1}^k \alpha_j Q^k(\lambda_{j_0+1}) \Big]   \frac{Q^k(\lambda_{j_0})}{Q^k(\lambda_{j_0+1})}= 1-  \frac{Q^k(\lambda_{j_0})}{Q^k(\lambda_{j_0+1})} >0.
\]
So that  $\alpha_{j_0}:= \frac{1- (\sum_{j=j_0+1}^k \alpha_j \big) Q^k(\lambda_{j_0})}{Q^k(\lambda_{j_0})}>0$ is uniquely defined.
\end{proof}

Now let us use this expression to conclude.  $\Sigma_1>0$ (trivially), and for all $j\leq k$, $\Sigma_j \geq 0$ would lead to $\sum_{j=1}^k \Delta^k(\lambda_i) >0$. Then, we can consider the truncated polynomials $p_k =\prod_{i=1}^k(x-\lambda_i)$ and $q_k =\prod_{i=1}^k(x-\mu_i)$, and do the simple fraction decomposition:
\[
\frac{q_k}{p_k}-1= \sum_{i=1}^k \Delta^k(\lambda_i) \frac{1}{x-\lambda_i}.
\]
 We can equate the leading coefficients on both sides which gives us the identity: $\sum_{i=1}^k \lambda_i - \sum_{i=1}^k \mu_i = \sum_{i=1}^k \Delta^k(\lambda_i) =0$. Which implies a contradiction and therefore a $\Sigma_j$ for $j \leq k$ has to be negative. 
\end{proof}

We derive easily the following continuity property.

\begin{lemma}\label{contmaj}
If    $p \succeq^s q$, then $up+(1-u)q  \succeq^s  vp+(1-v)q$ for all $u\geq v$ in $[0,1]$.
\end{lemma}
\begin{proof}
Denote the roots of $t\big(up+(1-u)q\big)+ (1-t)\big(vp+(1-v)q\big)$ by $\lambda_i(t,u,v)$. As:
\[
t\big(up+(1-u)q\big)+ (1-t)\big(vp+(1-v)q\big)= p \big(ut+v(1-t) \big) + q\big(1 -  \big(ut+v(1-t) \big)\big)
\]
and for $s\geq t$:
\[ us+v(1-s)\geq ut+v(1-t) ,
\] 
We obtain, using $p \succeq^s q$, for all $k \leq n$:
\[
\sum_{i=1}^k \lambda_i(s,u,v) \geq \sum_{i=1}^k \lambda_i(t,u,v),
\]
which proves that  $up+(1-u)q  \succeq^s  vp+(1-v)q$. 
\end{proof}
From this, we deduce the following characterization of strong majorization.
\begin{cor}[Strict monotonicity of strong majorization]\label{strictmon}
Assume $p$ and $q$ admit a common interlacer and have distinct roots, as well as $\sum_{i}^n\lambda_i= \sum_{i}^n\mu_i$.Then:
\[
   p \succeq^s q  \iff \text{ }\forall k<n,  \sum_{i=1}^k\lambda_i(t) \text{ is increasing (strictly) on } [0,1]. 
\]
\begin{proof}
Assume that a partial sum  $\sum_{i=1}^k\lambda_i(t) $ is not increasing, for $1<k<n$. It implies that there exist $s$ and $t$ such that $s>t$ and $\sum_{i=1}^k\lambda_i(s)= \sum_{i=1}^k\lambda_i(t)$, as strong majorization means the partial sums are nondecreasing. By Lemma~\ref{contmaj},  $p_s=sp+(1-s)q$ and $p_t=tp+(1-t)q$ are such that $p_s\succeq^s p_t$. Using Theorem~\ref{DiffMaj}, we get a contradiction. Therefore all the sums are increasing (strictly). 
\end{proof}
\end{cor}

Note that if strong majorization fails in this extreme case when partial sums of roots are equal, in can also happen in a neighborhood.

\end{document}